\documentclass[12pt]{amsart}

\usepackage{amssymb}

\usepackage{enumitem}

\usepackage{graphicx}

\makeatletter
\@namedef{subjclassname@2020}{%
  \textup{2020} Mathematics Subject Classification}
\makeatother

\usepackage[T1]{fontenc}

\newtheorem{theorem}{Theorem}[section]

\newtheorem{lemma}[theorem]{Lemma}

\newtheorem*{tha*}{{\textbf{Theorem A}}}

\newtheorem*{thb*}{{\textbf{Theorem B}}}

\theoremstyle{definition}

\newtheorem{remark}[theorem]{Remark}

\numberwithin{equation}{section}

\usepackage{mathrsfs, dsfont, url}

\newcommand{\vertiii}[1]{{\left\vert\kern-0.25ex\left\vert\kern-0.25ex\left\vert #1 
    \right\vert\kern-0.25ex\right\vert\kern-0.25ex\right\vert}}

\newcounter{smallromans}

\newenvironment{romanenumerate}
{\begin{list}{{\normalfont\textrm{(\roman{smallromans})}}}%
  {\usecounter{smallromans}\setlength{\itemindent}{0cm}%
   \setlength{\leftmargin}{5.5ex}\setlength{\labelwidth}{5.5ex}%
   \setlength{\topsep}{.5ex}\setlength{\partopsep}{.5ex}%
   \setlength{\itemsep}{0.1ex}}}%
{\end{list}}

\newcounter{smallromansdash}

{\end{list}}

\newcounter{bigromans} 
  {\end{list}}

\begin{document}


\baselineskip=17pt

\date{\today}
\title[Complementability of Banach spaces in the bidual]{Invariant means on Abelian groups capture complementability of Banach spaces\\ in their second duals}

\author{Adam P.~Goucher}
\address{Department of Pure Mathematics and Mathematical Statistics, University of Cambridge}
\email{goucher@dpmms.cam.ac.uk}

\author{Tomasz Kania}
\address{Institute of Mathematics, Czech Academy of Sciences, \v{Z}itn\'{a} 25, 115~67 Prague 1, Czech Republic and Institute of Mathematics, Jagiellonian University, {\L}ojasiewicza 6, 30-348 Krak\'{o}w, Poland}
\date{\today}
\email{kania@math.cas.cz, tomasz.marcin.kania@gmail.com}


\thanks{The second-named author acknowledges with thanks funding received from SONATA BIS no.~2017/26/E/ST1/00723.}
\begin{abstract}Let $X$ be a Banach space. Then $X$ is complemented in the bidual $X^{**}$ if and only if there exists an invariant mean $\ell_\infty(G, X)\to X$ with respect to a free Abelian group $G$ of rank equal to the cardinality of $X^{**}$, and this happens if and only if there exists an invariant mean with respect to the additive group of $X^{**}$. This improves upon previous results due to Bustos Domecq \cite{felix} and the second-named author \cite{kania}, where certain idempotent semigroups of cardinality equal to the cardinality of $X^{**}$ were considered, and answers a question of J.M.F.~Castillo (private communication) that was also considered in \cite{kania}. \emph{En route} to the proof of the main result, we endow the family of all finite-dimensional subspaces of an infinite-dimensional vector space with a structure of a~free commutative monoid with the property that the product of two subspaces contains the respective subspaces, which is possibly of interest in itself.\end{abstract}
\maketitle

\section{Introduction and the main result}

Every Banach space $X$ embeds canonically (linearly and isometrically) into its second dual space $X^{**}$ via the map $\kappa_X\colon X\to X^{**}$ defined by $\langle \kappa_X x, f\rangle = \langle f,x\rangle$ ($x\in X, f\in X^*$), so that it is customary to identify $X$ with a~closed subspace of $X^{**}$. The question of when is $X$ complemented in $X^{**}$ (that is, when does there exist a bounded linear projection from $X^{**}$ onto the image of $\kappa_X$) is a frequently recurring problem in Banach space theory as Banach spaces complemented in the bidual enjoy various, otherwise often unavailable, averaging properties such as the existence of vector valued invariant means with respect to amenable semigroups. Banach spaces complemented in their biduals include, besides reflexive spaces, dual spaces and weakly sequentially complete Banach lattices; $c_0$, the space of sequences convergent to 0, is the most prominent example of a space not complemented in its bidual.\smallskip

Let $X$ be a Banach space, $(S, \cdot)$ a semigroup, and $\lambda\geqslant 1$. A bounded linear operator $M\colon \ell_\infty(S,X)\to X$ of norm at most $\lambda$ such that for every $x\in X$, $s\in S$, and for all $f\in \ell_\infty(S,X)$ one has
\begin{romanenumerate}
\item $M(x\mathds{1}_S) = x$;
\item $M(f) = M({}_sf) =  M(f_s)$,\end{romanenumerate}
is called an $X$\emph{-valued invariant $\lambda$-mean} on $S$. Here, $\ell_\infty(S,X)$ denotes the Banach space of all bounded $X$-valued functions on $S$, $x\mathds{1}_S$ stands for the function constantly equal to $x$ on $S$, and ${}_sf(t) = f(s\cdot t)$ and $f_s(t) = (t\cdot s)$ ($s,t\in S$). When there is no need to emphasise the constant $\lambda$, $M$ is simply called an $X$\emph{-valued invariant mean on }$S$.\smallskip

A semigroup is \emph{amenable}, whenever it admits an $X$-valued 1-mean with respect to the one-dimensional space $X$ (namely, the scalar field). By the Markov--Kakutani fixed-point theorem, all commutative semigroups are ame\-na\-ble. In the literature, means invariant with respect to the additive monoid of natural numbers (with zero) are called \emph{Banach limits} (\emph{vector-valued Banach limits}, when $\dim X \geqslant 2$).\smallskip

Invariant means with respect to more complicated groups were employed by Pe{\l}czy\'{n}ski, who proved that for a Banach space $X$, if $Y\subseteq X$ is a closed subspace, which is complemented in $Y^{**}$ and there exists a Lipschitz map $r\colon X\to Y$ such that $r(y)=y$ for $y\in Y$, then $Y$ is (linearly) complemented in $X$ (\cite[pp.~61--62]{pelczynski}, see also \cite[Theorem 3.3]{benyamini}). 

 Suppose that $S$ is an amenable semigroup and $X$ is a Banach space that is complemented in $X^{**}$. Then one can construct an $X$-valued invariant {$\lambda$-mean} on $S$ as follows. Let $m$ be a~(scalar-valued) invariant mean on $S$. We define a map $\widetilde{M}\colon \ell_\infty(S, X)\to X^{**}$ by $$\langle \widetilde{M}f, \varphi\rangle = \big\langle m,  \langle \varphi, f(\boldsymbol{\cdot} )\rangle \big\rangle   \quad (\varphi \in X^*, f\in \ell_\infty(S, X)). $$
Since $\|m\|=1$, $\widetilde{M}$ is a norm-one bounded linear operator. If $P$ is a~projection from $X^{**}$ onto the canonical copy of $X$, then $$M = \kappa_X^{-1}P\widetilde{M}$$ is an $X$-valued invariant $\|P\|$-mean on $S$.\smallskip

In order to prove that the converse holds true as well, in \cite{felix, kania} the existence of an invariant mean on a certain idempotent semigroup $S$ was assumed; the semigroup $S$ comprises pairs $(F,\varepsilon)$, where $F\subset X^{**}$ is a~finite-dimensional subspace and $\varepsilon > 0$; the semigroup operation was defined as
$(F, \varepsilon) \cdot (G, \delta) = (F+G, \min\{\varepsilon, \delta\})$; related questions were also considered in \cite{radek}.\smallskip

In this note we consider the following question communicated privately to the second-named author by J.M.F. Castillo:

\begin{quote}
    \emph{Suppose that a Banach space $X$ admits an invariant mean with respect to every/some Abelian group. Must $X$ be complemented in $X^{**}$?}
\end{quote}\smallskip

The question arises from the observation that the semigroup $S$ used to prove the existence of a projection from $X^{**}$ onto $X$ is idempotent, so in a~sense, it is as far from being a (subset of a) group as possible. In the present paper we demonstrate that free Abelian groups of big enough rank suffice to prove the existence of a projection from the bidual.

\begin{tha*}\label{main}Let $X$ be a Banach space and $\lambda\geqslant 1$. Then the following assertions are equivalent.
\begin{romanenumerate}
\item\label{complement} $X$ is complemented in $X^{**}$ by a projection of norm at most $\lambda$;
\item\label{allamenable} for every amenable semigroup $S$ there exists an $X$-valued invariant $\lambda$-mean on $S$;
\item\label{cardinality} when $G$ is a free Abelian group of rank $|X^{**}|$, there exists an $X$-valued invariant $\lambda$-mean on $G$; 
\item\label{bidualcase} {there exists an $X$-valued invariant $\lambda$-mean on the additive group of $X^{**}$.}
\end{romanenumerate}\end{tha*}
The implication \eqref{complement} $\Rightarrow$ \eqref{allamenable}, has been already briefly explained and may be found, for example, in \cite[Theorem~1]{felix}. The implications \eqref{allamenable} $\Rightarrow$ \eqref{cardinality} and \eqref{allamenable} $\Rightarrow$ \eqref{bidualcase} are clear. We shall prove the implications \eqref{cardinality} $\Rightarrow$ \eqref{complement} and \eqref{bidualcase} $\Rightarrow$ \eqref{cardinality} in the subsequent section. \smallskip

Every Abelian group is a quotient of a free Abelian group, so by Lemma \ref{normal}(ii), clause (iii) in Theorem A is not weaker than the assertion that for every Abelian group $G$ of cardinality $X^{**}$ there exists an $X$-valued invariant $\lambda$-mean on $G$.\smallskip

The additive group of $X^{**}$ may seem large as already for separable Banach spaces that contain an isomorphic copy of $\ell_1$, it has the cardinality of the power set of the continuum; however, this is unavoidable. Indeed, it follows from \cite[Corollary 1.4]{kania} that the Banach space $X = \ell_\infty^c(\Gamma)$ of all countably supported bounded functions on an uncountable set $\Gamma$ (endowed with the supremum norm) admits invariant means with respect to all countable amenable semigroups, yet $X$ is not complemented in $X^{**}$ (\cite{pelsud}).\smallskip

\begin{remark}We remark in passing that the mere existence of an $X$-valued invariant mean with respect to \emph{some} commutative semigroup (that could be chosen as large as one wishes) is not sufficient for $X$ being complemented in $X^{**}$. To see this, fix an arbitrary Banach space $X$ (not necessarily complemented in $X^{**}$) and let $(S, \cdot)$ be a semigroup with an~element $0\in S$ such that $0\cdot s = s\cdot 0 = 0$ for all $s\in S$. Define $M\colon \ell_\infty(S,X)\to X$ by $Mf = f(0)$. Certainly, $M$ is a norm-one linear operator and $M(x \mathds{1}_S)=x$ for any $x\in X$. Moreover, for any $s\in S$, $M f_s = f_s(0) = f(0\cdot s) = f(0) = Mf = M{}_s f$. \end{remark}

It is known that if a separable Banach space has a monotone basis (or more generally, the Metric Approximation Property), then $X$ is 1-complemented in $X^{**}$ if and only if there exists an $X$-valued 1-mean on $\mathbb N$ (an $X$-valued Banach limit); see \cite[Corollary 4.2.6]{spaniards}. It follows from Lemma~\ref{groth} that this is equivalent to the existence of an $X$-valued invariant 1-mean on $\mathbb Z$. Thus, at least for very well-behaved separable Banach spaces complementability in the bidual can be indeed witnessed by the group of integers (the free abelian group of rank one).\medskip

The key result needed to establish Theorem~A is related to the possibility of turning the family of all finite-dimensional subspaces of $X^{**}$ into a~cancellative monoid.\smallskip

\begin{thb*} 

Let $V$ be an infinite-dimensional vector space over an arbitrary field. Denote by ${\rm Fin}\, V$ the set of all finite-dimensional subspaces of $V$. Then there exists a~binary operation $\ast$ on ${\rm Fin}\, V$ such that

\begin{romanenumerate}
\item $({\rm Fin}\, V, \ast)$ is a free commutative monoid;
\item for any $F,G \in {\rm Fin}\, V$ we have $F,G\subseteq F\ast G$.
\end{romanenumerate}
Consequently, the Grothendieck group of $({\rm Fin}\, V, \ast)$ is a free Abelian group.
\end{thb*}

It is quite clear that there is no counterpart of Theorem B for a finite-dimensional vector space $V$ as seen by taking three different hyperplanes $W_1, W_2, W_3$ in $V$ and noticing that $V=W_1\ast W_3 = W_2\ast W_3$, which makes cancellativity impossible to satisfy.\smallskip

It is noteworthy that the proof of Theorem B makes use of Zermelo's well-ordering principle and does depend on the chosen well-ordering of the set ${\rm Fin}\, X^{**}\setminus \{0\}$, however the resulting group is always free Abelian of rank $|X^{**}|$.

\section{Auxiliary facts}
We denote by $\mathbb N = \{0,1,2,3, \ldots\}$ the additive monoid of non-negative integers. As such, it is a submonoid of the group of integers $\mathbb Z$. A \emph{multiset} is formally a pair $(\Gamma, \varphi)$, where $\Gamma$ is an arbitrary non-empty set and $\varphi\colon \Gamma\to \{1,2,3\ldots \}$ is function counting the multiplicity of a given item in the multiset.\smallskip

An Abelian group $G$ is \emph{free}, when it is free as a module over $\mathbb Z$; as such $G$ is isomorphic to the direct sum $\mathbb Z^{(\alpha)}$ of $\alpha$-many copies of the group $\mathbb Z$. The number $\alpha$ is called the \emph{rank} of $G$. In particular, two uncountable free Abelian groups of the same cardinality are isomorphic. Besides free Abelian groups, we will require the notion of a \emph{free commutative monoid}, that is, a~monoid isomorphic to the direct sum $\mathbb N^{(\alpha)}$ of $\alpha$-many copies of $\mathbb N$ for some cardinal $\alpha.$\medskip

Let $(S,+)$ be a commutative semigroup. Then $S$ embeds into a group if and only if $S$ is cancellative. In the latter case, $S$ embeds into its \emph{Grothendieck group} $G(S)$, that is the group comprising equivalence classes of the relation $\sim$ on $S\times S$ given by $$(s_1, t_1)\sim (s_2, t_2) \iff s_1 + t_2 = t_1 + s_2\quad (s_1, s_2, t_1, t_2\in S).$$
Thus, formally $G(S) = \{s - t \colon s,t\in S\}$
and $0_S = 0_{G(S)}$ belongs to $S \cap (-S)$ if $S$ is already a~monoid. Of course, $\mathbb Z$ is the Grothendieck group of $\mathbb N$ and more generally $\mathbb{Z}^{(\alpha)}$ is the Grothendieck group of the free commutative monoid $\mathbb{N}^{(\alpha)}$ for any cardinal number $\alpha$.\smallskip

All commutative semigroups are amenable, yet subsemigroups of ame\-na\-ble groups need not be amenable. In the vector-valued case, we need to justify separately the possibility of passing to a subsemigroup, even in the commutative case. \smallskip

We shall require the following lemma concerning invariant means on groups that admit the infinite cyclic group as a quotient.

\begin{lemma}\label{trivial}
Let $X$ be a Banach space and let $G$ be a group with a surjective homomorphism $\theta\colon G \rightarrow \mathbb{Z}$. Suppose that $M\colon \ell_\infty(G,X)\to X$ is an $X$-valued invariant mean. Then $Mf = 0$ for any function $f$ supported on the kernel of $\theta$.
\end{lemma}

\begin{proof}
Assume otherwise. Let $g \in G$ be an arbitrary element satisfying $\theta(g) = 1$. Then for each $n \in \mathbb{N}$ and $t \in G$ with $\theta(t) \neq -n$, we have that $\theta(t g^n) \neq 0$ and, therefore, $f_{g^n}(t) = f(t g^n) = 0$. As such, the functions $\{ f_{g^n}\colon n \in \mathbb{N} \}$ have pairwise disjoint support.\smallskip

By the translation-invariance, $M(f_{g^n}) = M(f)$ for every $n \in \mathbb{N}$,
so $$M(f_g + f_{g^2} + f_{g^3} + \cdots + f_{g^n}) = n M(f).$$ As the summands have pairwise disjoint supports, the sum has the same supremum norm as $f$ itself (and, in particular, is constant as a function of $n$). On the other hand, the norm of $n M(f)$ grows without bound as $n \rightarrow \infty$, contradicting the continuity of $M$.
\end{proof}


\begin{lemma}\label{groth}Let $X$ be a Banach space and let $\lambda \geqslant 1$. If for some $\alpha$ there exists an~$X$-valued invariant $\lambda$-mean on $\mathbb{Z}^{(\alpha)}$, then there exists such a mean on $\mathbb{N}^{(\alpha)}$ too.\end{lemma}

\begin{proof}
Let $M\colon \ell_\infty(\mathbb{Z}^{(\alpha)}, X)\to X$ be an $X$-valued invariant $\lambda$-mean. For $t\in \mathbb{Z}^{(\alpha)}$ we define $|t|$ coordinate-wise, that is, $|t|(i) = |t(i)|$ for $i\in \alpha$. Let us consider the operator $M^\prime \colon \ell_\infty(\mathbb{N}^{(\alpha)}, X)\to X$ given by the formula $M^\prime(f) = M(f')$, where $f'(t) := f(|t|)$ for $t\in \mathbb{Z}^{(\alpha)}$ and $f\in \ell_\infty(\mathbb{Z}^{(\alpha)}, X)$. 
In particular, for the function $f = x \mathds{1}_{\mathbb{N}^{(\alpha)}}$ ($x\in X$), we have $$M^\prime f = M'(x\mathds{1}_{\mathbb{N}^{(\alpha)}})= M(x\mathds{1}_{\mathbb{Z}^{(\alpha)}}) = x.$$

Moreover, $f'$ and $f$ have the same supremum norm by definition; consequently, the operator norm of $M'$ is upper-bounded by the operator norm of $M$, which is in turn no greater than $\lambda$. To show that $M'$ is an $X$-valued invariant $\lambda$-mean, it remains to show that $M'$ is translation-invariant. Since every $g \in \mathbb{N}^{(\alpha)}$ can be written as a finite sum of `basis elements' $e_i$ where $i \in \alpha$, it suffices to show that for every $i \in \alpha$ and every $f \in \ell_\infty(\mathbb{N}^{(\alpha)}, X)$, we have $M'(f_{e_i}) = M'(f)$.\smallskip

Let $\pi_i\colon \mathbb{Z}^{(\alpha)} \rightarrow \mathbb{Z}$ be the `projection map' homomorphism which satisfies $\pi_i(e_i) = 1$ and $\pi_i(e_j) = 0$ for all $j \neq i$. For a predicate $\phi\colon \mathbb{Z} \rightarrow \{ \textrm{true}, \textrm{false} \}$ and $h \in \ell_\infty(\mathbb{Z}^{(\alpha)}, X)$, we define $h[\phi(\pi_i)]$ to be the restriction of $h$ to the values where the predicate is true:
$$ h[\phi(\pi_i)](x) := \begin{cases}
h(x), & \textrm{ if } \phi(\pi_i(x)), \\
0, & \textrm{ otherwise.}
\end{cases} $$
Then we can write:
$$ f' = f'[\pi_i \leqslant -2] + f'[\pi_i = -1] + f'[\pi_i = 0] + f'[\pi_i \geqslant 1] $$
and similarly:
$$ (f_{e_i})' = (f_{e_i})'[\pi_i \leqslant -1] + (f_{e_i})'[\pi_i \geqslant 0]. $$
Observe the following:
\begin{itemize}
    \item $f'[\pi_i \leqslant -2]$ is a translate of $(f_{e_i})'[\pi_i \leqslant -1]$, so $$M(f'[\pi_i \leqslant -2]) = M((f_{e_i})'[\pi_i \leqslant -1]);$$
    \item $f'[\pi_i \geqslant 1]$ is a translate of $(f_{e_i})'[\pi_i \geqslant 0]$, so $$M(f'[\pi_i \geqslant 1]) = M((f_{e_i})'[\pi_i \geqslant 0]);$$
    \item $M(f'[\pi_i = 0])$ is zero by Lemma~\ref{trivial};
    \item $M(f'[\pi_i = 1])$ is also zero, by Lemma~\ref{trivial} combined with translation-invariance of $M$.
\end{itemize}

By linearity of $M$, it follows that $M(f') = M((f_{e_i})')$, and therefore (by definition) $M'(f) = M'(f_{e_i})$. The result follows.
\end{proof}

It is a standard fact that subgroups and quotients of amenable discrete groups are amenable. Using exactly the same ideas one can prove that if a~Banach space admits an invariant mean with respect to a group, then so does it with respect to subgroups and quotients of the said group. 

\begin{lemma}\label{normal}Let $X$ be a Banach space and let $G$ be a group with a normal subgroup $H$. Suppose that there exists an  $X$-valued invariant $\lambda$-mean $M\colon \ell_\infty(G,X)\to X$ on $G$. Then,
\begin{romanenumerate}
\item there exists an $X$-valued invariant $\lambda$-mean on $H$;
\item there exists an $X$-valued invariant $\lambda$-mean on $G/H$;
\item if, moreover, $S$ and $T$ are isomorphic semigroups and there exists an $X$-valued invariant mean on $S$, then there exists such a mean on $T$ too.
\end{romanenumerate} \end{lemma}

\begin{proof}
In order to prove (i), let $\iota\colon \ell_{\infty}(H, X)\to \ell_\infty(G,X) = \ell_{\infty}(\bigsqcup_{i\in I} Hg_i, X)$ be given by $\iota(f) =  \bigsqcup_{i\in I} f_i$, where   $f_i(hg_i) = f(h)$ ($i\in I$). Here, $(g_i)_{i\in I}$ are representative elements in $G$ chosen so that $G = \bigsqcup_{i\in I} Hg_i$. Then $\widehat{M} = M \circ \iota$ is the sought invariant mean on $H$.\smallskip

For (ii), let $\pi\colon G\to G/H$ be the canonical quotient map. Then the map $\widehat{M}(f) = M( f\circ \pi)$ ($f\in \ell_\infty(G/H,X)$) is the sought invariant mean on $G/H$.\smallskip

Clause (iii) is quite trivial. Let $\theta\colon T\to S$ be an isomorphism and $M\colon \ell_\infty(S, X)\to X$ a~mean on $S$. Then $Nf = M(f\circ \theta)$ ($f\in \ell_\infty(T,X)$) is an invariant mean on $T$
\end{proof}

We shall require towards the proof the principle of local reflexivity due to Lindenstrauss and Rosenthal (\cite{lindros}), which we state here for the future reference.

\begin{theorem}\label{LR}Let $X$ be a Banach space and let $F\subset X^{**}$ be a finite-dimensional subspace. Then, for each $\varepsilon \in (0,1)$ there exists a linear map $P_F^\varepsilon \colon F\to \kappa_X(X)$ such that
\begin{romanenumerate}
\item $(1-\varepsilon)\|x\|\leqslant \|P_F^\varepsilon x\| \leqslant (1+\varepsilon)\|x\| \quad (x\in F)$;
\item $P_F^\varepsilon x = x$ for $x\in F\cap \kappa_X(X)$.
\end{romanenumerate}
\end{theorem}

\section{Proofs of the main results}
We start with the proof of Theorem B on which Theorem A is reliant. The proof proceeds by partitioning $S := \textrm{Fin } V \setminus \{ 0 \} $ (the set of all non-zero finite-dimensional subspaces of an infinite-dimensional vector space $V$) into disjoint sets $\mathcal{P}$ and $\mathcal{Q}$. Then, we construct a~bijection between $\textrm{Fin } V$ and the free commutative monoid $F_{\mathcal{P}}$, obtaining the operation $\ast$ by pulling back the monoid operation through this bijection.
\begin{proof}[Proof of Theorem B]

We begin by using Zermelo's well-ordering principle to construct a~bijection $\theta\colon S \rightarrow \alpha$, where $\alpha$ is a cardinal number (that is, an initial ordinal). This induces a well-ordering $\phi$ on the collection $F_S$ of finite multisets of elements from $S$. Specifically, if $A \in F_S$ contains elements $U_1, U_2, \dots, U_k$ with respective multiplicities $c_1, c_2, \dots, c_k$, then we define:
$$ \phi(A) := \omega^{\theta(U_1)} \cdot c_1 + \dots + \omega^{\theta(U_k)} \cdot c_k, $$
where without loss of generality we have taken $\theta(U_1) > \dots > \theta(U_k)$. The well-ordering $\phi$ on $F_S$ is known as \emph{colexicographical order}, and the order type is again $\alpha$.

$F_S$ is a graded monoid, where the grading $v(A) = c_1 + \dots + c_k$ is the total number of elements (counted with multiplicity) in the multiset. It is convenient to define, for each $n \in \mathbb{N}$, the set
$$ F_S^{(n)} := \{ A \in F_S\colon v(A) \geqslant n \}. $$

In particular, $F_S^{(0)} = F_S$, $F_S^{(1)}$ is the collection of non-empty multisets in $F_S$, and $F_S^{(2)}$ is the collection of non-empty non-singleton multisets in $F_S$.

For each $\beta \leqslant \alpha$, we define an injective partial function $g_{\beta}\colon F_S^{(2)} \rightarrow S$. Specifically:

\begin{itemize}
    \item When $\beta = 0$ is the zero ordinal, $g_{\beta}$ is the empty partial function.
    \item When $\beta = \gamma^{+}$ is a successor ordinal, $g_{\beta}$ is obtained from $g_{\gamma}$ by extending it with another element $g_{\beta}(A) = W$. $A$ is defined to be the first (according to the colexicographical order $\phi$) element $A \in F_S^{(2)}$ such that $A$ is not in the domain of $g_{\gamma}$ and no element of $A$ is in the image of $g_{\gamma}$. $W$ is an arbitrary element of $S$ such that:
    \begin{enumerate}
        \item For all $B \subsetneq A$ with $v(B) \geqslant 2$, $g_{\gamma}(B)$ is a subspace of $W$ (note that the colexicographical ordering on $F_S^{(2)}$ ensures that we have already defined $g_{\gamma}$ on all such subsets $B$);
        \item For all $U \in A$, $U$ is a subspace of $W$;
        \item $W$ occurs later in the well-ordering on $S$ induced by $\theta^{-1}$ than any element in $A$, or any element in any
        $B \in \textrm{dom}(g_{\gamma})$, or any element in the image
        of $g_{\gamma}$.
    \end{enumerate}
    It is necessary to show that these conditions are consistent.
    Conditions (1) and (2) merely specify that $W$ contains a particular finite-dimensional subspace $W'$ (the Minkowski sum of all $U \in A$ and all $g_\gamma(B)$ where $B \subsetneq A$ with $v(B) \geqslant 2$); as $V$ is infinite-dimensional, the collection of spaces $\{ W \in \textrm{Fin } V\colon W' \subseteq W \}$ has the same cardinality $\textrm{Fin } V$ itself. Owing to our choice of an initial ordinal for the well-ordering, condition (3) only eliminates a strictly smaller set of candidates for $W$, so the set of admissible $W$ is non-empty.
    \item When $\beta$ is a limit ordinal, $g_{\beta}$ is the union of the partial functions $g_{\gamma}$ (for all $\gamma < \beta$).
\end{itemize}


Let $\mathcal{Q}$ be the image of $g_{\alpha}$ and let
$\mathcal{P} := S \setminus \mathcal{Q}$. We \emph{claim} that the domain
of $g_{\alpha}$ is exactly $F_{\mathcal{P}}^{(2)}$.

To prove that
$F_{\mathcal{P}}^{(2)} \subseteq \textrm{dom}(g_{\alpha})$,
observe that for any set $A \in F_{\mathcal{P}}^{(2)}$, we have
that $A \in F_S^{(2)}$ and $A \cap \mathcal{Q} = \varnothing$; as
such, $A$ must have appeared at some point in the transfinite
induction and therefore belongs to the domain of $g_{\alpha}$.

In the other direction, suppose that $A \in \textrm{dom}(g_{\alpha})$
and $A \notin F_{\mathcal{P}}^{(2)}$. Then there exists an~element
$W \in A$ such that $W \in \mathcal{Q}$. Now, $A$ was introduced into
the domain of $g_{\alpha}$ at some successor ordinal $\beta = \gamma^{+}$
of the transfinite induction. By the definition of $A$, $W$ cannot be in
the image of $g_{\gamma}$, so $W$ must have appeared at some later step
in the transfinite induction. However, that contradicts condition (3),
which forces $W$ to be later in the well-ordering (and therefore not
equal to) any element of $A$.

Consequently, $F_{\mathcal{P}}^{(2)} = \textrm{dom}(g_{\alpha})$ so
we have a total function $g_{\alpha} : F_{\mathcal{P}}^{(2)} \rightarrow
\mathcal{Q}$.


We extend $g_{\alpha}$ to a bijection $f : F_{\mathcal{P}} \rightarrow \textrm{Fin } V$ by setting:

\begin{itemize}
    \item $f(\varnothing) := \{0\}$;
    \item $f(\{V\}) := V$ for all singleton sets $\{V\}$ where $V \in \mathcal{P}$;
    \item $f(A) := g_{\alpha}(A)$ if $v(A) \geqslant 2$.
\end{itemize}

At this point, we define a binary operation $\ast\colon \textrm{Fin } V \times \textrm{Fin } V \rightarrow \textrm{Fin } V$ by $$U_1 \ast U_2 := f(f^{-1}(U_1) \sqcup f^{-1}(U_2)),$$ where $\sqcup$ is the usual `disjoint union' operation that makes $F_{\mathcal{P}}$ into a free commutative monoid. This immediately establishes part (i) of the proposition, that $({\rm Fin}\, V, \ast)$ is a free commutative monoid.\smallskip

What remains to be shown is part (ii) of Theorem B, that for any subspaces $F,G \in {\rm Fin}\, V$ we have $F,G \subseteq F \ast G$. We can assume that neither $F$ nor $G$ is the trivial space $\{0\}$, as this is the identity of the operation $\ast$ and the statement would trivially hold.\smallskip

Consequently, $v(f^{-1}(F)) \geqslant 1$ and $v(f^{-1}(G)) \geqslant 1$, so $v(f^{-1}(F \ast G)) \geqslant 2$. This means that $A := f^{-1}(F \ast G)$ is in $F_{\mathcal{P}}^{(2)}$ and therefore appeared (together with the space $W := F \ast G$) in the transfinite induction. We shall show that $F$ is a subspace of $W := F \ast G$:

\begin{itemize}
    \item If $v(f^{-1}(F)) = 1$, then $f^{-1}(F)$ is the singleton set $\{F\}$. As such, $F \in A$, and therefore $F$ is a subspace of $W$ by condition (2) in the transfinite induction.
    \item If $v(f^{-1}(F)) \geqslant 2$, then $f^{-1}(F) = B \subsetneq A$ and $F = f(B)$ is a subspace of $W$ by condition (1) in the transfinite induction.
\end{itemize}

By symmetry, it follows that $G$ is also a subspace of $F \ast G$, establishing the truth of Theorem B.
\end{proof}

We are now ready to prove Theorem A.

\begin{proof}[Proof of Theorem A] We are now to prove the implications \eqref{cardinality} $\Rightarrow$ \eqref{complement} and \eqref{bidualcase} $\Rightarrow$ \eqref{cardinality}.\smallskip

Without loss of generality we may assume that $X$ is infinite-dimensional. In order to prove the former implication, let us consider the free commutative monoid
$$ S = {\rm Fin}\, X^{**} \oplus \{2^{-k}\colon k\in \mathbb N\},$$
where ${\rm Fin}\, X^{**}$ is the monoid of finite-dimensional subspaces of $X^{**}$ considered in Theorem B and $\{2^{-k}\colon k\in \mathbb N\}$ is the monoid generated by $2$ in the multiplicative group of non-zero real numbers. By Theorem B, $S$ is a free commutative monoid and as such, it is isomorphic to $\mathbb N^{(|X^{**}|)}$. \smallskip

Assume that there exists an $X$-valued invariant $\lambda$-mean on the Gro\-then\-dieck group $G(S)$ of $S$, bearing in mind that $G(S)$ is isomorphic to $\mathbb Z^{(|X^{**}|)}$ and $S$ is isomorphic to $\mathbb N^{(|X^{**}|)}$. By Lemma~\ref{groth} (and Lemma~\ref{normal}(iii)), we may fix an $X$-valued invariant $\lambda$-mean on $S$, $M\colon \ell_\infty(S, X)\to X$.\smallskip

Using the principle of local reflexivity (Theorem~\ref{LR}), for each $\varepsilon = 2^{-k}$ with $k\geqslant 1$ and $F\in {\rm Fin}\, X^{**}$, we fix $P^{\varepsilon}_{F}$ as in the statement. We also set $P^1_F = 0$, the zero operator on $F\in {\rm Fin}\, X^{**}$. \smallskip

For an element $x\in X^{**}$, we define a function $f^x\colon S \to X$ by
$$f^x(F, \delta) = \left\{\begin{array}{ll}P^{\min\{\delta, 1/2\}}_{F}x, & x\in F \\ 0,& x\notin F\end{array}\right.\quad \big((F,\delta)\in S\big).$$
Then $f^x \in \ell_\infty(G, X)$ with $\|f^x\| \leqslant 2\|x\|$. \smallskip

 We define $P\colon X^{**}\to X$ by $$Px = M(f^x)\quad (x\in X^{**})$$ and \emph{claim} that it is the sought linear projection. Certainly, for $x\in \kappa_X(X)$, we have $$Px = M(f^x) =  M (P^{\min\{\boldsymbol{\cdot}, 1/2\}}_{\boldsymbol{\cdot}}x)\!  = M (P^{\min\{\boldsymbol{\cdot}, 1/2\}}_{\boldsymbol{\cdot}\ast[x]}x) = M (x \mathds{1}_{S})  = x,$$
 where $[x]$ stands for the line spanned by $x$. Also, $P$ is linear. Indeed, homogeneity is clear. To see that $P$ is additive, we compute:
$$\begin{array}{lcl}P(x+y)& =& M(f^{x+y})\\
&=&   M \big(P^{\min\{\boldsymbol{\cdot}, 1/2\}}_{\boldsymbol{\cdot}}(x+y)\big)\\
& = & M \big(P^{\min\{\boldsymbol{\cdot}, 1/2\}}_{\boldsymbol{\cdot}\ast [x,y]}x\big) + M \big(P^{\min\{\boldsymbol{\cdot}, 1/2\}}_{\boldsymbol{\cdot}\ast [x,y]}y\big)\\
& = & M \big(P^{\min\{\boldsymbol{\cdot}, 1/2\}}_{\boldsymbol{\cdot}}x\big) + M \big(P^{\min\{\boldsymbol{\cdot}, 1/2\}}_{\boldsymbol{\cdot}}y\big)\\
&=& Px + Py,\end{array}$$
where $[x,y]$ is the linear span of $x$ and $y$ ($x,y\in X^{**})$. Here, we used Theorem B from which it follows that for any $F\in {\rm Fin}\, X^{**}$, $[x,y]\subset F\ast [x,y]$. Finally, for any $x\in X^{**}$, by the translation-invariance of $M$, we have
$$\|Px\| = \|M(f^x)\| \leqslant \lambda \cdot \inf_{k\geqslant 1}\|f^x_{([x], 1/2^k)}\| \leqslant \lambda\cdot \inf_{k\geqslant 1} \sup_{F\in {\rm Fin}\, X^{**}} \|P^{1/2^k}_F x\| \leqslant \lambda\|x\|, $$
which proves that indeed $X$ is $\lambda$-complemented in $X^{**}$.\medskip

In order to prove \eqref{bidualcase} $\Rightarrow$ \eqref{cardinality}, we observe that $X^{**}$ is, in particular, a~$\mathbb Q$-vector space, and being infinite-dimensional (as an $\mathbb R$-vector space), it is thus isomorphic to the direct sum $\mathbb Q^{(|X^{**}|)}$. As such, it contains a subgroup isomorphic to $\mathbb Z^{(|X^{**}|)}$. Thus, if there exists an $X$-valued invariant mean on $X^{**}$, by Lemma~\ref{normal}, there is an $X$-valued invariant mean on a free Abelian group of rank $|X^{**}|$. \end{proof}

\subsection{Countable amenable groups} As already explained, separable Banach spaces with the Metric Approximation Property are 1-complemented in their bidual if and only if they admit an invariant 1-mean on the group of integers. If one replaces the Metric Approximation Property with the $\gamma$-Bounded Approximation Property (with parameter $\gamma \geqslant 1$), then complementation in the bidual is still characterised by the existence of an~invariant mean on the group of integers but the correspondence between the norm of the mean and the upper bound for the norm of a projection from the bidual is impeded by the parameter $\gamma$. Nonetheless, we do not know the answer to the following question.\smallskip

\begin{quote}\emph{Let $X$ be a separable Banach space. Suppose that for some countably infinite amenable group $G$ there exists an $X$-valued invariant mean on $G$. Is $X$ complemented in the bidual $X^{**}$?}
    
\end{quote}

\subsection*{Acknowledgements} The second-named author wishes to express thanks to Rados{\l}aw {\L}ukasik for numerous conversations concerning the problem. Morever, we thank the referee for his or her extremely thorough report.

\end{document}